\documentclass[preprint,10pt]{amsart}
\usepackage[T1]{fontenc}
\usepackage[english]{babel}
\usepackage{amssymb}
\usepackage{mathtools}
\usepackage{amsmath, empheq}
\usepackage[applemac]{inputenc}
\usepackage{amsthm}
\usepackage{amsaddr}
\usepackage{enumerate}
\usepackage{verbatim}
\newcommand{\partn}[1]{{\smallskip \noindent \textbf{#1.}}}
\newcommand{\partnn}[1]{{\smallskip \noindent \textbf{#1.}}}

\DeclareMathOperator{\ord}{ord}

\DeclareMathOperator{\id}{Id}
\DeclareMathOperator{\Q}{\mathbb{Q}}

\DeclareMathOperator{\Z}{\mathbb{Z}}

\DeclarePairedDelimiter{\ang}{\langle}{\rangle}
\numberwithin{equation}{section}

\newcommand\myeq{\mathrel{\overset{\makebox[0pt]{\mbox{\normalfont\tiny\sffamily IA}}}{=}}} 

\begin{document}

\renewcommand{\thefootnote}{\roman{footnote}} 
\newtheorem{mydef}{Definition} 
\newtheorem{thm}{Theorem} 
\newtheorem{MYthm}{Theorem}
\renewcommand*{\theMYthm}{\Alph{MYthm}}
\newtheorem{MYlemma}{Lemma}
\renewcommand*{\theMYlemma}{\Alph{MYlemma}}
\newtheorem{MYprop}{Proposition}
\renewcommand*{\theMYprop}{\Alph{MYprop}}
\newtheorem{example}{Example} 
\newtheorem{conj}{Conjecture}
\newtheorem{prop}{Proposition}
\newtheorem{corr}{Corollary} 
\newtheorem{lemma}{Lemma} 
\theoremstyle{remark}
\newtheorem{rem}{Remark}
\newtheorem*{prob}{Problem}

\pagenumbering{roman}

\title{Characterization of 2-ramified power series}
\author{Jonas Nordqvist}
\email{jonas.nordqvist@lnu.se}
\address{Department of Mathematics, Linn\ae us University, V\"{a}xj\"{o}, Sweden}
\maketitle

\begin{abstract}
In this paper we study lower ramification numbers of power series tangent to the identity that are defined over fields of positive characteristics $p$. Let $g$ be such a series, then $g$ has a fixed point at the origin and the corresponding lower ramification numbers of $g$ are then, up to a constant, the degree of the first  non-linear term of $p$-power iterates of $g$. The result is a complete characterization of power series $g$ having ramification numbers of the form ${2(1+p+\dots+p^n)}$. Furthermore, in proving said characterization we explicitly compute the first significant terms of $g$ at its $p$th iterate.

\vspace{4mm}
\end{abstract}

\setcounter{page}{1}
\pagenumbering{arabic}

\section{Introduction}
Iteration of power series is an important part of arithmetic dynamics \cite{Silverman2007, AnashinKhrennikov2009, Shparlinski2007}, which is a very active area of research within the general field of dynamical systems.
In this article we are interested in power series defined over fields of positive characteristic $p$. In particular, the degree of the first non-linear term of $p$-power iterates of power series tangent to the identity, the so called lower ramification numbers of such series.
Recent results by Lindahl and Rivera-Letelier \cite{Lindahl2013, LindahlRiveraLetelier2015, LindahlRiveraLetelier2013} show that there is a connection between the geometric location of periodic points of power series over ultrametric fields and lower ramification numbers.

Throughout the paper let $p$ be a prime number and $k$ a field of characteristic $p$, also let $g \in k[[\zeta]]$ be a power series of the form \begin{equation}\label{g}g(\zeta) = \zeta+\cdots.\end{equation} The \emph{order} of a nonzero power series $g$ is the lowest degree of its non-zero terms and we denote this by $\ord(\cdot)$. We put $\ord(0):= +\infty$. Let $g^n(\zeta)$ denote the $n$-fold composition of itself. The \emph{lower ramification number} $i_n$ of a power series $g$ is defined as \begin{equation}\label{ramifnumber}i_n(g) := \ord\left(\frac{g^{p^n}(\zeta)-\zeta}{\zeta}\right).\end{equation}

A famous theorem of Sen \cite{Sen1969},\footnote{Alternative proofs are given in \cite{Lubin1995, LindahlRiveraLetelier2013}.} often referred to as Sen's theorem, states that if $g^m \not = \id$, for all integers $m\geq 1$ then \begin{equation*}i_n(g) \equiv i_{n-1}(g) \pmod{p^n}.\end{equation*}
Given $i_0(g) = 1$ we then have \begin{equation}\label{raminfnumber}i_n(g) \geq 1 + p +\dots+p^n.\end{equation} Rivera-Letelier \cite[Example 3.19]{RiveraLetelier2003} gave a complete characterization of power series for which (\ref{raminfnumber}) holds with equality. Keating \cite{Keating1992} proved a similar result for the special case that $p=3$.

\begin{mydef}\label{def2ramif}
Let $p$ be a prime and $k$ a field of characteristic $p$. Furthermore, let $g$ be a power series in  $k[[\zeta]]$. Then $g$ is said to be \emph{$2$-ramified} if its sequence of lower ramification numbers is of the form 
\begin{equation}\label{brbr}i_n(g) = 2(1+p+\dots+p^n).\end{equation} 
\end{mydef}

In this paper we give a complete characterization of all power series of the form (\ref{g}) for which Definition \ref{def2ramif} is satisfied.

\begin{thm}\label{maintheorem}Let $p$ be an odd prime and let $k$ be a field of characteristic $p$. 
Let $g \in k[[\zeta]]$ be a power series of the form 
\begin{equation*}\label{polyQ}g(\zeta) = \zeta\left(1  + \sum_{j=1}^{+\infty} a_{j}\zeta^{j}\right).\end{equation*} 
Then $g$ is 2-ramified if and only if $a_1 = 0$, $a_2\not= 0$ and 
\begin{equation}\label{result}3/2a_2^3 + a_3^2 -a_2a_4 \not= 0.\end{equation}
\end{thm}

\begin{rem}
Note that for $p=2$, $g$ can be written of the form $\zeta + a_2\zeta^{p+1} + \cdots$, which implies $i_0(g) = p$. From \cite[Corollary 1]{LaubieSaine1998} we know that if $p$ divides $i_0(g)$ then $i_n(g) = i_0(g)p^{n}$ for all integers $n\geq0$. Therefore, Theorem \ref{maintheorem} is only stated for odd primes. 
\end{rem}

Our primary technical result is the computation of the first significant terms of the $p$th iterate of $g$. 

\begin{prop}\label{mainprop}
Let $p$ be a prime and $k$ field of characteristic $p$. Let $g \in k[[\zeta]]$ be a power series of the form 
\[g(\zeta) = \zeta\left(1 + a_2\zeta^2 + a_3\zeta^3 + a_4\zeta^4\right) \mod \langle \zeta^6 \rangle,\] 
then 
\begin{equation}\label{thepropeq}g^p(\zeta) - \zeta \equiv a_2^{p-2}\left(3/2a_2^3+a_3^2-a_2a_4\right)\zeta^{2p+3} \mod \langle \zeta^{2p+4} \rangle.\end{equation}
\end{prop}

In combination with the following result from Laubie and Sa\"{i}ne 
this computation yields sufficient information to prove Theorem \ref{maintheorem}. 

\begin{lemma}[\cite{LaubieSaine1998}, Corollary 1]\label{laubiesainelemma}
Let $p$ be a prime and $k$ be a field of characteristic $p$. Moreover, let $g \in k[[\zeta]]$, be a power series of the form (\ref{g}). If $p \nmid i_0(g)$ and $i_1(g) < (p^2-p+1)i_0(g)$, then \[i_n(g) = i_0(g) + \frac{p^n-1}{p-1}(i_1(g) - i_0(g)),\] for all integers $n \geq0$.
\end{lemma}

Lemma \ref{laubiesainelemma} is a restatement of the last statement of Corollary 1 in \cite{LaubieSaine1998}. 

Our approach of calculating the $p$th iterate of $g$ given in Proposition \ref{mainprop} is similar to methods used in \cite{LindahlRiveraLetelier2015, LindahlRiveraLetelier2013}. Proofs of Theorem \ref{maintheorem} and Proposition \ref{mainprop} will be given in \S3.

\subsection{Implication}
In \cite{LindahlRiveraLetelier2015, LindahlRiveraLetelier2013} the authors study the relation between the lower ramification numbers and the geometric location of periodic points of ultrametric dynamical systems. In their results several important discoveries concerning lower ramification numbers have been made. The authors are especially interested in parabolic power series not necessarily tangent to the identity. Let $\gamma \in k$ such that $\gamma^q = 1$, and $f\in k[[\zeta]]$, be a power series of the form $f(\zeta) = \gamma\zeta + \cdots$. Then $f$ is said to be \emph{minimally ramified}\footnote{The notion of minimal ramification was first introduced for wildly ramified automorphisms, i.e. $\gamma = 1$ by \cite{LaubieMovahhediSalinier2002} where it corresponds to ramification numbers of the form $1+p+\cdots+p^n$.} if \[i_n(f^q) = q(1+p+\dots+p^n).\]

In \cite[Theorem C]{LindahlRiveraLetelier2015} a characterization of minimally ramified power series is given. However, for the case that $\gamma =-1$, i.e. $q=2$, minimal ramification corresponds to 2-ramification, in the sense that $f$ is minimally ramified if and only if $f^2$ is 2-ramified. This implies that Theorem \ref{maintheorem} provides an alternative proof for this special case, where $f^2 = g$. In fact straightforward computation yields the following corollary from Theorem \ref{maintheorem}.

\begin{corr}\label{corrminuz}
Let $p$ be an odd prime and $k$ a field of characteristic $p$. Furthermore let $f \in k[[\zeta]]$ be of the form \begin{equation*}\label{fminuszeta} f(\zeta) = \zeta\left(-1 + \sum_{j=1}^{+\infty} a_j \zeta^{j}\right).\end{equation*} Then $f^2(\zeta)$ is 2-ramified if and only if \begin{equation*}\label{corrresult}(a_1^2+a_2) (11 a_1^4+25 a_1^2 a_2+12 a_1 a_3+6 a_2^2+4 a_4)  \not= 0\end{equation*}
\end{corr}

The proof of this corollary follows from Theorem \ref{maintheorem} and is given in \S3.1, and a consequence of the corollary is the following interesting example.

\begin{example}\label{minuzpoly}
Let $p$ be an odd prime and $k$ a field of characteristic $p$. Let $f \in k[\zeta]$ be the polynomial $f(\zeta) = -\zeta + \zeta^2$, 
then $f^2$ is 2-ramified if and only if $p\not=11$.
\end{example}

\subsection{Related works}
The relation between lower ramification numbers and arithmetic dynamics over ultrametric fields is one of the motivations for this study. However, ramification numbers have also been considered in different contexts. In the study of the potential sequences of ramification numbers, Keating \cite{Keating1992} used the relation between the ramification numbers and abelian extensions of $k((t))$. Laubie and Sa\"{i}ne \cite{LaubieSaine1997, LaubieSaine1998} could later improve these results by applying Wintenberger's theory on fields of norms \cite{Wintenberger2004}. In \cite{LaubieMovahhediSalinier2002} the authors study Lubin's conjecture \cite{Lubin1994}, on the relation between wildly ramified power series and formal groups.

\subsection{Acknowledgments}
I would like to thank Karl-Olof Lindahl for fruitful discussions and support during this work, and for introducing me to the topic. His comments and guidance certainly helped me improve the presentation, and clear my thoughts about these matters. I would also like to thank Juan Rivera-Letelier and the referee for helpful comments on the paper.

\section{Preliminaries}
Given a ring $R$ and an element $a\in R$ we let $\langle a \rangle$ denote the ideal of $R$ generated by $a$.

Throughout the paper for any nonnegative integer $n$ let $n!!$ denote the \emph{double factorial} of $n$. Let $0!! = 1!! = 1$, and for integers $n\geq 2$, we have \[n!! = n(n-2)!!.\]

Let $(k, |\cdot|)$ be an ultrametric field, and let $\mathcal{O}_k$ denote the ring of integers of $k$, and $\mathfrak{m}_k$ its maximal ideal. Let $\widetilde{k}:=\mathcal{O}_k/\mathfrak{m}_k$ be the residue field of $k$. Furthermore, denote the projection in $\widetilde{k}$ of an element $a$ of $\mathcal{O}_k$ by $\widetilde{a}$; it is the reduction of $a$. The reduction of a power series $g \in \mathcal{O}_k[[\zeta]]$ is the power series $\widetilde{g}(\zeta) \in \widetilde{k}[[\zeta]]$ whose coefficients are the reductions of the corresponding coefficients of $g$. Moreover, let $\Q_p$ be the $p$-adic numbers and let  $\Z_p$ denote its ring of integers.

\section{Characterization of 2-ramified power series}
In this section we prove our main results. As mentioned in the introduction Theorem \ref{maintheorem} is a consequence of Proposition \ref{mainprop}, which also will be proved within this section. However, to prove this in turn we will need several lemmas and the proofs of Theorem \ref{maintheorem} and Proposition \ref{mainprop} will follow thereafter.

\begin{lemma}\label{sumident} Let $p$ be an odd prime. For each integer $n\geq1$ let $\mathcal{R}_n$ and $\mathcal{T}_n$ in $\Q_p$ be defined by
\begin{equation}\mathcal{R}_n:=(2n-1)!!\sum_{r = 1}^{n}\left[\prod_{j=r+1}^{n}\frac{2j}{2j-1}\right],\end{equation}
and
\begin{equation}\mathcal{T}_n:=(2n+1)!!\sum_{j=1}^n\frac{(2j)!!}{(2j+1)!!}.
\end{equation}
Then \[\mathcal{R}_n = (2n+1)!! - (2n)!! \text{ and }\mathcal{T}_n = (2n+2)!! - 2(2n+1)!!.\] Moreover, $\widetilde{\mathcal{R}}_p = \widetilde{\mathcal{T}}_p = 0$.
\end{lemma}

\begin{proof} The last consequence of the lemma follows from the first by simply putting $n=p$.

\partnn{Proof of $\mathcal{R}_n$}
We proceed by induction in $n$ to see that the identity is valid. For $n = 2$ it holds and we prove that it holds for arbitrarily chosen $n$, but first note that 
\[\sum_{r = 1}^{n}\left[\prod_{j=r+1}^{n}\frac{2j}{2j-1}\right] = \sum_{r = 1}^{n}\frac{(2n)!!(2r-1)!!}{(2n-1)!!(2r)!!}.\] Now we proceed by induction in $n$ and study $n+1$
\begin{align*}
\sum_{r = 1}^{n+1}\frac{(2n+2)!!(2r-1)!!}{(2n+1)!!(2r)!!} 
&= \frac{(2n+2)!!(2n+1)!!}{(2n+1)!!(2n+2)!!} + \sum_{r = 1}^{n}\frac{(2n+2)!!(2r-1)!!}{(2n+1)!!(2r)!!} \\
&= 1 + \frac{2n+2}{2n+1}\sum_{r = 1}^{n}\frac{(2n)!!(2r-1)!!}{(2n-1)!!(2r)!!} \\
&\myeq 1 + \frac{2n+2}{2n+1}\left((2n+1) - \frac{(2n)!!}{(2n-1)!!}\right) \\
&= 1 + (2n+2) - \frac{(2n+2)!!}{(2n+1)!!} \\
&= (2n+3) - \frac{(2n+2)!!}{(2n+1)!!},
\end{align*}
by multiplying with $(2n+1)!!$ we get the proposed identity. 

\partnn{Proof of $\mathcal{T}_n$}
A straightforward computation shows that the identity holds for $n=1$. We proceed by induction in $n$. Assume that the lemma holds for $n\geq1$. Then 
\begin{align*}\sum_{j=1}^{n+1}\frac{(2j)!!}{(2j+1)!!} &= \frac{(2n+2)!!}{(2n+3)!!} + \sum_{j=1}^n\frac{(2j)!!}{(2j+1)!!}\\ &\myeq \frac{(2n+2)!!}{(2n+3)!!} + \frac{(2n+2)!! - 2(2n+1)!!}{(2n+1)!!} \\
 &= \frac{(2n+3 + 1)(2n+2)!! - 2(2n+3)!!}{(2n+3)!!} \\
&= \frac{(2n+4)!!-2(2n+3)!!}{(2n+3)!!},
\end{align*} which completes the induction step, and we finish the proof by multiplying both sides with $(2n+3)!!$.
\end{proof}

\begin{lemma}\label{sumidentfinitefield}
Let $p$ be an odd prime and let $\alpha,\beta$ be integers. For every integer $n\geq1$ let $\mathcal{S}_n(\alpha,\beta)$ in $\Q_p$ be defined by
\begin{equation}\mathcal{S}_n(\alpha,\beta) := (2n+1)!!\sum_{j=1}^n\frac{\alpha j + \beta}{2j+1}.\end{equation}
Then $\mathcal{S}_p(\alpha,\beta) \in \Z_p$ and $\widetilde{\mathcal{S}}_p(\alpha,\beta) = \alpha/2 - \beta$.
\end{lemma}

\begin{proof} Each term of $\mathcal{S}_p(\alpha,\beta)$ will contain a factor $p$ and thereby be an element in $p\Z_p$, except for the $j$th term where $j=(p-1)/2$ because then the factor $p$ in the numerator will vanish due to a factor $p$ in the denominator.

Hence, the only term of $\mathcal{S}_p(\alpha,\beta)$ not in $p\Z_p$ is 
\begin{align}\label{tppedpp}
\frac{(2p+1)!!}{p}(\alpha((p-1)/2) + \beta) &=(2p+1)\cdots(p+2)(p-2)\cdots3\cdot1((p-1)\alpha/2 + \beta).
\end{align}
This is certainly an element in $\Z_p$. Put \[ \mathcal{C} = \frac{(2p+1)!!}{p}.\] 
Then 
\begin{align*}\mathcal{C} &= (2p+1)(2p-1)\cdots(p+2)(p-2)\cdots3\cdot1 \\&= (p + p + 1)(p + p - 1)\cdots(p+4)(p+2)(p-2)\cdots3\cdot1 \\
&\equiv (p+1)(p-1)\cdots4\cdot2\cdot(p-2)!!\pmod{p}\\
&\equiv (p+1)!!(p-2)!! \pmod{p}\\
&\equiv (p+1)(p-1)!\pmod{p}.\end{align*}
By Wilson's theorem we know that $(p-1)!\equiv -1 \pmod{p}$ and from this we deduce 
\begin{equation}\label{csj}\widetilde{\mathcal{C}} = -1.\end{equation}
From (\ref{tppedpp})  and (\ref{csj}) we then obtain
\begin{align*}
\widetilde{\mathcal{S}}_p(\alpha,\beta) &= \alpha/2 - \beta .
\end{align*}
This completes the proof of Lemma \ref{sumidentfinitefield}.
\end{proof}

An important part of the proof of Proposition \ref{mainprop} is utilization of the following lemma, which is a generalization of a lemma from \cite{Elaydi2005} to hold for any field $k$.

\begin{lemma}\cite[\S 1.2]{Elaydi2005}\label{thmDIFF}
Let $k$ be a field. Furthermore let $f,g : \mathbb{Z}^+ \mapsto k$, and $y_0 \in k$. Given a nonhomogeneous difference equation \[y_{n+1} = f(n)y_{n} + g(n), y_{n_0} = y_0\] where $n_0 \in [0,n]$. The general solution to the difference equation is given by 
\begin{equation}\label{elaydieq}
y_{n} = \left[\prod_{j=n_0}^{n-1}f(j)\right]y_0 + \sum_{r = n_0}^{n-1}\left[\prod_{j=r+1}^{n-1}f(j)\right]g(r)
.\end{equation}
\end{lemma}

\begin{proof}
This will be proved using induction.
Assuming that (\ref{elaydieq}) holds for some $n$, we now prove it for $n+1$.
We have  \[y_{n+1} = f(n)y_n + g(n).\]
Substitution of $y_n$ with the induction assumption yields
\begin{align*}
y_{n+1} &= f(n)\left[\left[\prod_{j=n_0}^{n-1}f(j)\right]y_0 + \sum_{r = n_0}^{n-1}\left[\prod_{j=r+1}^{n-1}f(j)\right]g(r)\right] + g(n) \\
&= f(n)\left[\prod_{j=n_0}^{n-1}f(j)\right]y_0 + f(n)\sum_{r = n_0}^{n-1}\left[\prod_{j=r+1}^{n-1}f(j)\right]g(r) \\
&= \left[\prod_{j=n_0}^{n}f(j)\right]y_0 + \sum_{r = n_0}^{n}\left[\prod_{j=r+1}^{n}f(j)\right]g(r),
\end{align*}
which completes the induction step. 
\end{proof}


\begin{proof}[Proof of Proposition \ref{mainprop}] 
The proof of this proposition is divided into three parts. In the first part we define a recurrence relation $\Delta_m$ with the property that $\Delta_p = g^p(\zeta)-\zeta$, a method also used in e.g. \cite{RiveraLetelier2003}, \cite{LindahlRiveraLetelier2013}, and find the difference equations that defines the first three significant terms in $\Delta_m$. In the second part we solve these difference equations for an arbitrarily chosen $m$, and in the last part we determine the coefficient of the first significant term in $\Delta_p$ and hence of $g^p(\zeta)-\zeta$.

\partn{Part 1. Finding the difference equations} 
Analogous to \cite{LindahlRiveraLetelier2013} for $m= 1$ we define the recurrence relation $\Delta_1(\zeta) := g(\zeta) - \zeta$ and for $m \geq 2$ \[\Delta_m(\zeta) := \Delta_{m-1}(g(\zeta)) - \Delta_{m-1}(\zeta).\] Note that $\Delta_p(\zeta) = g^p(\zeta) - \zeta$. For technical reasons we define $G_\infty := \Q_p[x_2,x_3,\dots]$, and for each integer $\ell \geq 1$ put $G_\ell:=\Q_p[x_2,x_3,\dots, x_\ell]$. Moreover we consider the power series $\widehat{g} \in G_\infty[[\zeta]]$ defined as \[\widehat{g}(\zeta) := \zeta(1 + x_2\zeta^2 + x_3\zeta^3 + x_4\zeta^4) \mod \langle \zeta^6 \rangle.\] 
For $m=1$ we define the relation $\widehat{\Delta}_1(\zeta) := \widehat{g}(\zeta)-\zeta$ and for each integer $m\geq 2$ \[\widehat{\Delta}_m(\zeta) := \widehat{\Delta}_{m-1}(\widehat{g}(\zeta))-\widehat{\Delta}_{m-1}(\zeta).\]  
Defined in this way there is a clear relation between $g(\zeta)$ and $\widehat{g}(\zeta)$ and thus  between $\Delta_m$ and $\widehat{\Delta}_m$. In the last part of the proof we exploit this relation to find the coefficients of $g^p(\zeta) - \zeta$.

Concerning $\widehat{\Delta}_2(\zeta)$, we have 
\begin{align*}\widehat{\Delta}_2(\zeta) &= \widehat{\Delta}_1(\widehat{g}(\zeta)) - \widehat{\Delta}_1(\zeta)\\  
&=x_2(\zeta +x_2\zeta^3 +x_3\zeta^4 +x_4\zeta^5+x_5\zeta^6+ \cdots)^3 \\
&\hspace{4mm}+x_3(\zeta +x_2\zeta^3 +x_3\zeta^4 +x_4\zeta^5+x_5\zeta^6 + \cdots)^4  \\
&\hspace{4mm}+x_4(\zeta +x_2\zeta^3 +x_3\zeta^4+x_4\zeta^5 +x_5\zeta^6 + \cdots)^5 \\
&\hspace{4mm}+x_5(\zeta +x_2\zeta^3 +x_3\zeta^4+x_4\zeta^5 +x_5\zeta^6 + \cdots)^6 + \dots - \Delta_1(\zeta)\\ 
&\equiv x_2\zeta^3 + 3x_2^2\zeta^5 + 3x_2x_3\zeta^6 + \left(\binom{3}{2}x_2^3+3x_2x_4\right)\zeta^7 +x_3\zeta^4 + 4x_2x_3\zeta^6 + 4x_3^2\zeta^7  \\
&\hspace{4mm}+x_4\zeta^5 + 5x_2x_4\zeta^7+x_5\zeta^6 +x_6\zeta^7 - \widehat{\Delta}_1(\zeta) \mod \langle\zeta^8\rangle\\ 
&\equiv  3x_2^2\zeta^5 + 7x_2x_3\zeta^6 + \left(\binom{3}{2}x_2^3 + 4x_3^2 + 8x_2x_4\right)\zeta^7 \mod \langle\zeta^8\rangle. 
\end{align*}
We will see that it is necessary to keep track of the three first significant terms since $\ord(\widehat{\Delta}_{m+1}(\zeta)) = \ord(\widehat{\Delta}_m(\zeta)) + 2$ for $m \in \{1,\dots,p-1\}$. Note that neither the $x_5$- nor the $x_6$-term is affecting the first three significant terms.

Let $x = (x_2,x_3,\dots)$, more generally for a given $m \in \{1,\dots,p\}$ we have \begin{equation}\label{coeffisarna}\widehat{\Delta}_m(\zeta) = A_m(x)\zeta^{2m+1} + B_m(x)\zeta^{2m+2} + C_m(x)\zeta^{2m+3} + \cdots,\end{equation} where $A_m(x), B_m(x),C_m(x) \in G_\infty$. Throughout the rest of this proof let $A_m := A_m(x), B_m := B_m(x)$ and $C_m := C_m(x)$ unless otherwise specified. Defined in this manner the polynomials can be described by the following system of linear difference equations represented as follows
\begin{equation}\label{rec}
\begin{bmatrix}
 x_2(2m + 1) & 0 & 0 \\
 x_3(2m + 1) & x_2(2m+2) & 0 \\
(x_2^2m + x_4) (2m + 1) & x_3(2m+2) & x_2(2m + 3)
 \end{bmatrix}
\begin{bmatrix}
 A_{m} \\
 B_{m} \\
 C_{m} 
\end{bmatrix}
 =
\begin{bmatrix}
 A_{m+1} \\
 B_{m+1} \\
 C_{m+1} 
 \end{bmatrix},
 \end{equation}
with the initial conditions $(A_1, B_1, C_1) = (x_2, x_3,x_4)$.
To see that this actually describes the situation we study $\widehat{\Delta}_{m+1}(\zeta)$ and obtain
\begin{align*}
\widehat{\Delta}_{m+1}(\zeta) &= \widehat{\Delta}_{m}(\widehat{g}(\zeta)) - \widehat{\Delta}_{m}(\zeta) \\
&= \widehat{\Delta}_{m}(\zeta + x_2\zeta^3+x_3\zeta^4 + x_4\zeta^5 + x_5\zeta^6 + \cdots) - \widehat{\Delta}_{m}(\zeta)\\
&= A_m(\zeta + x_2\zeta^3 + x_3\zeta^4 + x_4\zeta^5+ x_5\zeta^6 + \cdots)^{2m+1}\\
&\hspace{4mm} + B_m(\zeta + x_2\zeta^3 + x_3\zeta^4 + x_4\zeta^5+ x_5\zeta^6 + \cdots)^{2m+2}\\
&\hspace{4mm}+ C_m(\zeta + x_2\zeta^3 + x_3\zeta^4 + x_4\zeta^5+ x_5\zeta^6 + \cdots)^{2m+3}+\cdots -\widehat{\Delta}_{m}(\zeta) \\
&\equiv A_mx_2(2m+1)\zeta^{2m+3} + A_mx_3(2m+1)\zeta^{2m+4} + A_m\left(x_2^2\binom{2m+1}{2}  + x_4(2m+1)\right)\zeta^{2m+5}\\
&\hspace{4mm}+ B_mx_2(2m+2)\zeta^{2m+4} + B_mx_3(2m+2)\zeta^{2m+5} + C_mx_2(2m+3)\zeta^{2m+5} \mod \langle\zeta^{2m+6}\rangle\\ 
&\equiv A_mx_2(2m+1)\zeta^{2m+3} + (A_mx_3(2m+1) + B_mx_2(2m+2))\zeta^{2m+4} \\
&\hspace{4mm} + (A_m((2m+1)(x_2^2m +x_4)) + B_mx_3(2m+2) + C_mx_2(2m+3))\zeta^{2m+5} \mod \langle\zeta^{2m+6}\rangle.
\end{align*}
This in fact implies that $A_m, B_m, C_m \in G_4$. The next step of the proof is to find a closed form expression of these polynomials after $m$ iterations

\partn{Part 2. Solving the difference equations}
In this section we discuss the solutions to the difference equations (\ref{rec}) given in the previous section. First note that since $\Q_p$  is a field and the equations in (\ref{rec}) are linear there are unique solutions $\{A_m\}_{m\geq1}$, $\{B_m\}_{m\geq1}$ and $\{C_m\}_{m\geq1}$ respectively. Also note that all three difference equations are first order, and except for $A_m$ they are nonhomogeneous.

We now apply Lemma \ref{thmDIFF} to solve the equations. We start by considering the top equation in (\ref{rec}) \[A_{m+1} = x_2(2m+1)A_{m},\quad A_1 = x_2.\] Considering Lemma \ref{thmDIFF} we obtain the solution \begin{equation}\label{solutionA}A_m = \left[\prod_{j=1}^{m-1}x_2(2j+1)\right]A_1 = x_2^{m-1}(2m-1)!!x_2 = x_2^m(2m-1)!!.\end{equation} Now we have the solution to our first difference equation and insertion into $B_m$ in (\ref{rec}) yields 
\begin{align}\label{recBupdate}
B_{m+1} &= x_2(2m+2)B_{m} + x_3(2m+1)A_{m} \notag\\ 
&= x_2(2m+2)B_{m} + x_3(2m+1)x_2^{m}(2m-1)!! \notag\\
 &= x_2(2m+2)B_{m} + x_2^{m}x_3(2m+1)!!.
 \end{align}

Now we can see that $B_m$ is in fact a nonhomogeneous difference equation. We utilize substitution to obtain a simpler expression. Assuming $x_3\neq0$ this yields
\begin{equation}\label{Bsub}
B_{m+1}^* = \frac{B_{m+1}}{x_2^{m}x_3(2m+1)!!}.
\end{equation} 
Insertion of (\ref{Bsub}) into (\ref{recBupdate}) yields
\[x_2^{m}x_3(2m+1)!!B_{m+1}^* = x_2^{m}x_3(2m+2)(2m-1)!!B^*_{m} + x_2^{m}x_3(2m+1)!!,\]
hence \[B_{m+1}^* = \frac{2m+2}{2m+1}B^*_{m} + 1.\]
Note that $B_1 = x_3$ implies that $B^*_1 = 1$ from (\ref{Bsub}). Now we can solve $B_m^*$ using Lemma \ref{thmDIFF} and we obtain
\begin{align*}
 B_{m}^* &= \left[\prod_{j=1}^{m-1}\frac{2j+2}{2j+1}\right]B^*_1 + \sum_{r = 1}^{m-1}\left[\prod_{j=r+1}^{m-1}\frac{2j+2}{2j+1}\right] \\
&= \sum_{r=1}^m\left[\prod_{j=r+1}^m \frac{2j}{2j-1}\right].
\end{align*} 
Substitution of $B_m^*$ in (\ref{Bsub}) for $B_m$ together with the definition of $\mathcal{R}_m$ in Lemma \ref{sumident} yields 
\begin{align}\label{Bsolve}
B_{m} &= x_2^{m-1}x_3(2m-1)!!\sum_{r=1}^m\left[\prod_{j=r+1}^m \frac{2j}{2j-1}\right] \notag \\
&= x_2^{m-1}x_3\mathcal{R}_m \notag \\
 &= x_2^{m-1}x_3\left((2m+1)!!-(2m)!!\right).
 \end{align}
Hence, we have a solution for $B_m$ as well which means that we can express $C_m$ (the last equation in (\ref{rec})) as a nonhomogeneous difference equation depending solely on $C_{m}$ and $m$. Inserting the solutions for $A_m$ and $B_m$, (\ref{solutionA}) and (\ref{Bsolve}) respectively into (\ref{rec}) with initial value $C_1 = x_4$ and for $m\geq1$ we have  
\begin{align}\label{recCnonhomo} C_{m+1} &= x_2^{m}(2m-1)!!((2m+1)(x_2^2m + x_4)) \\ 
&\hspace{4mm} + 2(m+1)x_2^{m-1}x_3^2((2m+1)!!-(2m)!!) + x_2(2m+3)C_{m}\notag\\
&= x_2^{m+2}m(2m+1)!! + x_2^{m}x_4(2m+1)!!\notag\\
&\hspace{4mm} + 2(m+1)x_2^{m-1}x_3^2((2m+1)!!-(2m)!!) + x_2(2m+3)C_{m}.\notag
\end{align}
Since this is a linear difference equation we simplify this problem by splitting this equation into three separate equations. For $m \geq 1$  we define $D_m, E_m$ and $F_m$ as 
\begin{equation}\label{rec2}
 \begin{bmatrix}
 x_2(2m+3) D_{m} \\
 x_2(2m+3) E_{m} \\
 x_2(2m+3) F_{m} 
 \end{bmatrix}
 +
 \begin{bmatrix}
 d(m) \\
e(m) \\
f(m) 
 \end{bmatrix}
 =
 \begin{bmatrix}
 D_{m+1} \\
 E_{m+1} \\
 F_{m+1} 
 \end{bmatrix},
 \end{equation}
with the initial conditions $(D_1, E_1, F_1) = (0, x_4, 0)$.

Furthermore, let \begin{align*}
d(m) &= x_2^{m+2}m(2m+1)!!\\
e(m) &= x_2^{m}x_4(2m+1)!!\\
f(m) &= 2(m+1)x_2^{m-1}x_3^2((2m+1)!!-(2m)!!),
\end{align*} and thereby equation (\ref{recCnonhomo}) is satisfied. This can be seen as follows 
\begin{align*}D_{m+1} + E_{m+1} + F_{m+1} &=  x_2(2m+3)D_{m} + d(m) + x_2(2m+3)E_{m} \\
&\hspace{4mm}+ e(m) + x_2(2m+3)F_{m} + f(m) \\
&= x_2(2m+3)(D_{m} + E_{m} + F_{m}) + d(m) + e(m) + f(m) \\
&= x_2(2m+3)(C_{m}) + d(m) + e(m) + f(m)\\
&= C_{m+1}.
\end{align*} By solving $D_m$, $E_m$ and $F_m$ separately we can retrieve the solution to $C_m$. 

We start by finding the solution for \begin{equation}\label{recD}D_{m+1} = x_2(2m+3)D_{m} + x_2^{m+2}m(2m+1)!!.\end{equation} Analogous with (\ref{Bsub}) we utilize substitution \begin{equation*}\label{subD}D_{m+1}^* = \frac{D_{m+1}}{x_2^{m+2}(2m+3)!!}.\end{equation*} Insertion into (\ref{recD}) yields \begin{equation*}D_{m+1}^* = D_{m}^* + \frac{m}{2m+3},\end{equation*} and since $D^*_1 = 0$ this means that for every iteration of the recursion the fraction term will add on. This yields a solution of the form \[D^*_m = \sum_{j=1}^{m-1}\frac{j}{2j+3} =  \sum_{j=1}^{m}\frac{j-1}{2j+1} .\] Changing variable back to $D_m$ we get that \begin{equation}\label{DSolved}D_m = x_2^{m+1}(2m+1)!!\sum_{j=1}^m\frac{j-1}{2j+1}.\end{equation}

The difference equation for $E_m$ is given by \begin{equation}\label{regE} E_{m+1} = x_2(2m+3)E_{m} + x_2^{m}x_4(2m+1)!!.\end{equation} Substitution gives us \begin{equation*}\label{subE}E_{m+1}^* = \frac{E_{m+1}}{x_2^{m}(2m+3)!!},\end{equation*} which then yields \begin{equation}\label{emmoebler} E_{m+1}^* = E_{m}^* + \frac{x_4}{2m+3}.\end{equation} By utilizing Lemma \ref{thmDIFF} the solution to (\ref{emmoebler}) is given by,
\begin{align*} 
E_m^* &= \frac{x_4}{3} + \sum_{j=1}^{m-1}\frac{x_4}{2j+3} = x_4\sum_{j=1}^{m}\frac{1}{2j+1}, 
\end{align*} 
which means that the solution to (\ref{regE}) is
\begin{equation} \label{Esolved}E_m = x_2^{m-1}x_4(2m+1)!!\sum_{j=1}^{m}\frac{1}{2j+1}.\end{equation}

The difference equation for $F_m$ is given by \begin{equation}\label{recF}F_{m+1} = x_2(2m+3)F_{m} + 2(m+1)x_2^{m-1}x_3^2((2m+1)!!-(2m)!!).\end{equation} Assume that $x_3\neq0$, and furthermore as in the previous case for $D_m$ we use substitution, and obtain \begin{equation*}\label{subF}F_{m+1}^* = \frac{F_{m+1}}{x_2^{m-1}x_3^2(2m+3)!!}.\end{equation*} Inserting this into (\ref{recF}) yields \[F_{m+1}^* = F_{m}^* + (2m+2)\left(\frac{(2m+1)!! - (2m)!!}{(2m+3)!!}\right) \iff F_{m+1}^* = F_{m}^* + \frac{2m+2}{2m+3} - \frac{(2m+2)!!}{(2m+3)!!}.\] By the same reasoning as above, since $F_1^* = 0$, we have \[F_m^* =\sum_{j=1}^{m-1} \left[\frac{2j+2}{2j+3} - \frac{(2j+2)!!}{(2j+3)!!}\right]= \sum_{j=1}^{m}\left[\frac{2j}{2j+1} - \frac{(2j)!!}{(2j+1)!!}\right],\] which means that the solution to our original equation is \begin{equation}\label{solvedF}F_m = x_2^{m-2}x_3^2(2m+1)!!\sum_{j=1}^{m}\left[\frac{2j}{2j+1} - \frac{(2j)!!}{(2j+1)!!}\right].\end{equation} Consequently, by summing (\ref{DSolved}), (\ref{Esolved}) and (\ref{solvedF}) we obtain
\begin{align}\label{CM}C_m &= D_m + E_m + F_m\\ 
&= x_2^{m+1}(2m+1)!!\sum_{j=1}^m\frac{j-1}{2j+1}\notag \\
&\hspace{4mm} + x_2^{m-1}x_4(2m+1)!!\sum_{j=1}^{m}\frac{1}{2j+1} \notag\\ 
&\hspace{4mm}+ x_2^{m-2}x_3^2(2m+1)!!\sum_{j=1}^{m}\left[\frac{2j}{2j+1} - \frac{(2j)!!}{(2j+1)!!}\right].\notag
\end{align}Together with (\ref{solutionA}) and (\ref{Bsolve}) we now have closed form expressions of the coefficients in (\ref{coeffisarna}).


\partn{Part 3. Determining the coefficients of $g^p(\zeta) - \zeta$}
Recall that \[\widehat{g}^p(\zeta)-\zeta = A_p\zeta^{2p+1}+B_p\zeta^{2p+2} + C_p\zeta^{2p+3} \mod \ang{\zeta^{2p+4}}.\] It follows from  (\ref{solutionA}) and (\ref{Bsolve}) that $A_p,B_p\in \Z_p[x_2,x_3,x_4]$. In particular, \begin{equation}\widetilde{A}_p = \widetilde{B}_p = 0.\end{equation}
Using the definitions of $\mathcal{S}_p$ and $\mathcal{T}_p$ from Lemma \ref{sumident}  and \ref{sumidentfinitefield} together with (\ref{CM}) we have
\begin{align*} C_p &= x_2^{p+1}\mathcal{S}_p(1,-1) + x_2^{p-1}x_4\mathcal{S}_p(0,1) + x_2^{p-2}x_3^2\Big(\mathcal{S}_p(2,0) - \mathcal{T}_p\Big).
\end{align*} 
Also note that by  Lemma \ref{sumident} and \ref{sumidentfinitefield} we have 
\[\widetilde{\mathcal{S}}_p(1,-1) = 3/2,\quad \widetilde{\mathcal{S}}_p(0,1) = -1,\quad \widetilde{\mathcal{S}}_p(2,0) = 1, \text{ and } \quad\widetilde{\mathcal{T}}_p = 0.\]
Consequently, \[\widetilde{C}_p(x_2,x_3,x_4) = x_2^{p-2}(3/2x_2^3+x_3^2-x_2x_4).\]
Finally, for each $i \geq 1$ we specialize each variable $x_i$ to $a_i$ and obtain
\begin{align}\label{solvCP}
\widetilde{C}_p(a_2,a_3,a_4) 
&= a_2^{p-2}\left(3/2a_2^3 + a_3^2 - a_2a_4\right).
\end{align}
We conclude that
\begin{align*}g^p(\zeta)-\zeta &= \widetilde{C}_p(a_2,a_3,a_4)\zeta^{2p+3} \mod \ang{\zeta^{2p+4}}\\
&= a_2^{p-2}(3/2a_2^3+a_3^2-a_2a_4)\zeta^{2p+3} \mod \ang{\zeta^{2p+4}}.
\end{align*} This completes the proof of Proposition \ref{mainprop}.
\end{proof} 

\begin{proof}[Proof of Theorem \ref{maintheorem}]
Assuming that $a_1 = 0$, $a_2 \neq 0$ and $3/2a_2^3+a_3^2-a_2a_4 \not= 0$, then we obtain by the assumption that $i_0(g) = 2$, and by Proposition \ref{mainprop} that ${i_1(g) = 2(1+p)}$, together with \cite[Corollary 1]{LaubieSaine1998} this implies that $g$ is 2-ramified. Conversely, assuming that $g$ is 2-ramified then by Proposition \ref{mainprop} we obtain $3/2a_2^3+a_3^2-a_2a_4 \not= 0$. We also recall that if $g$ is 2-ramified then $i_0(g) = 2$, which in turn implies that $a_1 = 0$ and $a_2 \neq0$.
\end{proof}

\subsection{Proof of implication}
We give a proof to Corollary \ref{corrminuz}, stated in \S1.1.
\begin{proof}[Proof of Corollary \ref{corrminuz}]
We compute $f^2(\zeta)$ explicitly. Recall that \[f(\zeta) = \zeta(-1 + a_1\zeta + a_2\zeta^2 + a_3\zeta^3 + a_4\zeta^4) \mod \langle \zeta^6 \rangle,\] which yields
\[f^2(\zeta) = \zeta  -2(a_1^2+a_2)\zeta^3 + (a_1^3+a_1a_2)\zeta^4 + (3a_2^3-6a_1a_3-a_1^2a_2-2a_4)\zeta^5 \mod \langle \zeta^6\rangle.\]
Using Theorem \ref{maintheorem} we have that $f^2$ is 2-ramified if and only if (\ref{result}) holds. This yields
\begin{align*}\label{fraccorr}
&\frac{1}{2}(-2a_1^2-2a_2)^{p-2}(3(-2a_1^2-2a_2)^3+2(a_1^3+a_1a_2)^2\\ 
&\hspace{4mm}-2(-2a_1^2-2a_2)(3a_2^3-6a_1 a_3-a_1^2a_2-2a_4)) \notag \\
&=-(-2a_1^2-2a_2)^{p-2}(a_1^2+a_2) (11 a_1^4+25 a_1^2 a_2+12 a_1 a_3+6 a_2^2+4 a_4) \notag \\ 
&=2^{p-2}(a_1^2+a_2)^{p-1}(11 a_1^4+25 a_1^2 a_2+12 a_1 a_3+6 a_2^2+4 a_4) ,\notag
\end{align*} 
which proves our corollary.
\end{proof}

\newpage

\bibliographystyle{alpha} 
\bibliography{bib}

\end{document}